\documentclass[a4paper,10pt]{amsart}

\usepackage{amssymb}
\usepackage{bbm}
\usepackage{mathrsfs}
\usepackage[all]{xy}
\usepackage[T1,T5]{fontenc}

\normalsize

\newtheoremstyle{bracket}{1ex}{2ex}{\rm}{}{\bfseries}{}{0.8em}{\thmnumber{(#2)}}
\newtheoremstyle{thm}{1ex}{2ex}{\itshape}{}{\bfseries}{}{0.9em}{\thmnumber{(#2)}\thmname{ #1}\thmnote{ (#3)}}
\newtheoremstyle{example}{1ex}{2ex}{\rm}{}{\bfseries}{}{0.8em}{\thmnumber{(#2)}\thmname{ #1}}

\theoremstyle{bracket}
\newtheorem{no}{}[section]

\theoremstyle{thm}
\newtheorem{prop}[no]{Proposition}
\newtheorem{cor}[no]{Corollary}
\newtheorem{thm}[no]{Theorem}

\theoremstyle{example}
\newtheorem{exa}[no]{Example}

\DeclareMathOperator{\ke}{Ker}
\DeclareMathOperator{\spec}{Spec}
\DeclareMathOperator{\face}{face}
\DeclareMathOperator{\cone}{cone}
\DeclareMathOperator{\rk}{rk}
\DeclareMathOperator{\depth}{depth}
\DeclareMathOperator{\diff}{Diff}
\DeclareMathOperator{\ob}{Ob}
\DeclareMathOperator{\pr}{pr}

\newcommand{\N}{\mathbbm{N}}
\newcommand{\Z}{\mathbbm{Z}}
\newcommand{\Q}{\mathbbm{Q}}
\newcommand{\R}{\mathbbm{R}}
\newcommand{\dfgl}{\mathrel{\mathop:}=}
\newcommand{\Id}{{\rm Id}}
\newcommand{\sig}{\Sigma}
\newcommand{\red}{{\rm red}}
\newcommand{\fleq}{\preccurlyeq}
\newcommand{\alg}{{\sf Alg}}
\newcommand{\ann}{{\sf Ann}}
\newcommand{\C}{{\sf C}}
\newcommand{\D}{{\sf D}}
\newcommand{\mon}{{\sf Mon}}
\newcommand{\sch}{{\sf Sch}}
\newcommand{\eps}{\varepsilon}
\newcommand{\lsq}{\hskip1pt\raisebox{.13ex}{\rule{.55ex}{.55ex}\hskip1pt}}
\newcommand{\sq}{\hskip2pt\raisebox{.225ex}{\rule{.8ex}{.8ex}\hskip2pt}}
\newcommand{\hm}[3]{{\rm Hom}_{#1}(#2,#3)}
\newcommand{\hmf}[2]{{\sf Hom}(#1,#2)}
\newcommand{\fun}{\mathbbm{F}_1}
\newcommand{\M}{\mathbbm{M}}
\newcommand{\xm}{X_{\mathbbm{M}}}
\newcommand{\tm}{t_{\mathbbm{M}}}
\newcommand{\xs}{X_{\sig}}
\newcommand{\ts}{t_{\sig}}

\newcommand{\snf}{\renewcommand{\thefootnote}{*}\footnotetext{The author was supported by the Swiss National Science Foundation.}}

\begin{document}

\title{The geometry of toric schemes\protect\snf}
\author{Fred Rohrer}
\address{Universit\"at T\"ubingen, Fachbereich Mathematik, Auf der Morgenstelle 10, 72076 T\"u\-bingen, Germany}
\email{rohrer@mail.mathematik.uni-tuebingen.de}

\begin{abstract}
Geometric properties of schemes obtained by gluing algebras of monoids, including separation and finiteness properties, irreducibility, normality, catenarity, dimension, and Serre's properties $(S_k)$ and $(R_k)$, are investigated. This is used to show how the geometry of a toric scheme over an arbitrary base is influenced by the geometry of the base.
\end{abstract}

\maketitle


\section*{Introduction}

During the last forty years, the theory of toric varieties was generalised in several directions. But the generalisation that seems to be the most natural and the most important -- the \textit{study of toric varieties from a scheme-theoretical point of view} -- was never actually carried out, although it is necessary for attacking fundamental questions such as whether or not the Hilbert scheme of a toric variety exists. It is clear that to do this one has to be able to make arbitrary base changes. Hence, instead of considering toric varieties over an algebraically closed field $k$ one needs to study \textit{toric schemes,} i.e., ``toric varieties over arbitrary base schemes''. To this end, also starting with a fan $\sig$, we perform the same construction as for toric varieties but replace throughout $k$ by an arbitrary ring $R$; this yields an $R$-scheme, and by gluing we can generalise this to an arbitrary base scheme $S$. So, we end up with an $S$-scheme $\ts(S)\colon\xs(S)\rightarrow S$, called \textit{the toric scheme over $S$ associated with $\sig$.} Its construction is functorial in $S$ and compatible with base change. While $\xs(S)$ still shares a lot of nice properties, being flat and of finite presentation over $S$, it can be as ugly as its base $S$. For example (and supposing $\sig\neq\emptyset$), we will see below that it is separated, irreducible, normal, universally catenary, or Cohen-Macaulay if and only if $S$ is so, and that even if $\sig$ is regular then $\xs(S)$ is regular if and only if $S$ is so. In particular, the Weil divisor techniques that are often used on toric varieties are in general \textit{not available on a toric scheme.}

Toric schemes in this sense were mentioned briefly in \cite{dem} (for regular fans and mainly in case $R=\Z$), but besides this, no literature about toric schemes seems to be available. (The equivalent description of toric varieties in terms of torus operations in the classical case was generalised in \cite{kkms} to base rings that are discrete valuation rings. However, already there the two descriptions -- via fans and via torus operations -- are no longer equivalent. For details about this divergence and also the case of more general valuation rings we refer the reader to the recent preprint \cite{gubler} by Gubler.)

\smallskip

For greater clarity, a generalisation of the above construction is described in this article. Namely, in Section \ref{sec2} we study functorial properties of algebras of monoids and introduce a general construction that turns a certain projective system $\M$ of monoids into an $S$-scheme $t_{\M}(S)\colon X_{\M}(S)\rightarrow S$ for every scheme $S$. Toric schemes will be defined as a special case of this construction, which is moreover related to $\fun$-schemes in the sense of Deitmar (cf.~\ref{f1}). In Section \ref{sec3} we show that $t_{\M}(S)$ has certain properties and apply this to see how the geometry of $S$ influences the geometry of $\xm(S)$ and vice versa. We define toric schemes in Section \ref{sec5} and apply the results from Section \ref{sec3} to obtain a first description of the geometry of toric schemes, before we extend it with results about regularity specific to toric schemes.

Our results point out where the geometry of toric schemes differs from the geometry of toric varieties, and hence give a first glimpse at the difficulties that are met when studying toric schemes. Still, it is the author's hope that this article will help to extend the ``remarkably fertile testing ground for general theories'' provided by toric varieties (as Fulton puts it in \cite{ful}) from the narrow setting of varieties to the more natural setting of schemes.

\smallskip

The material in this article is part of the author's Dissertation \cite{diss} (available on his homepage), to which the reader is referred for more details.

\smallskip

\textit{Notations and conventions:} In general we use the terminology of Bourbaki's \textit{\'El\'ements de math\'ematique} and Grothendieck's \textit{\'El\'ements de g\'eom\'etrie alg\'ebrique}. Monoids are understood to be additively written and commutative, rings are understood to be commutative, and algebras are understood to be commutative, unital and associative. We use the categories $\hmf{\C}{\D}$ (functors from $\C$ to $\D$ for categories $\C$ and $\D$), $\C_{/S}$ and $\C^{/S}$ (objects in $\C$ over or under $S$, respectively, for a category $\C$ and $S\in\ob(\C)$), $\mon$ (monoids), $\ann$ (rings), $\alg(R)$ ($R$-algebras for a ring $R$), $\sch$ (schemes), and $\sch_{/R}\dfgl\sch_{/\spec(R)}$ (for a ring $R$).


\section{Gluing algebras of monoids}\label{sec2}

Functors of algebras of monoids are the objects of study in this section. We briefly review their definition and some base change and further commutation results, in particular with respect to formation of rings of fractions. Then, a general construction of schemes from certain projective systems of monoids, dubbed \textit{openly immersive,} is treated; toric schemes will be obtained as a special case of this in Section \ref{sec5}. The reader will notice similarities to $\fun$-schemes in the sense of Deitmar (\cite{deitmar1}), and we will make this connection more precise at the end of this section.

\begin{no}\label{alg20}
Let $R$ be a ring. The forgetful functor $\alg(R)\rightarrow\mon$ mapping an $R$-algebra onto its underlying multiplicative monoid has a left adjoint that maps a monoid $M$ onto its algebra over $R$, denoted by $s(R,M)\colon R\rightarrow R[M]$. This is also functorial in $R$, hence we get a diagram of categories $$\xymatrix@R3pt{&\ar@{=>}@<-1.5ex>[dd]^s&\\\ann\times\mon\ar@/^3ex/[rr]^{\pr_1}\ar@/_3ex/[rr]_{\bullet[\lsq]}&&\ann.\\&&}$$ Composition with $\spec\colon\ann^{\circ}\rightarrow\sch$ and setting $t\dfgl\spec\circ s$ yields a diagram $$\xymatrix@R3pt{&&\\\ann^{\circ}\times\mon^{\circ}\ar@/^3ex/[rr]^{\spec\circ\pr_1}\ar@/_3ex/[rr]_{\spec(\bullet[\lsq])}&&\sch.\\&\ar@{=>}@<1.5ex>[uu]_t&}$$ For a monoid $M$, the $R$-module underlying $R[M]$ is the free $R$-module with basis the set underlying $M$, yielding a map $\exp_{R,M}\colon M\rightarrow R[M]$; if no confusion can arise we set $e_m\dfgl\exp_{R,M}(m)$ for $m\in M$.
\end{no}

\begin{no}\label{alg40}
Let $F\in\ob(\hmf{\ann}{\ann}^{/\Id_{\ann}})$. For a ring $R$ and a monoid $M$, the maps $$M\rightarrow\hm{R}{F(R)}{F(R)[M]},\;m\mapsto(x\mapsto xe_m)$$ and $$M\rightarrow R[M]\otimes_RF(R),\;m\mapsto e_m\otimes 1$$ induce mutually inverse morphisms in $\alg(F(R))$ between $R[M]\otimes_RF(R)$ and $F(R)[M]$ that are natural in $R$ and $M$. So, there is a canonical isomorphism $$\bullet[\sq]\otimes_{\bullet}F(\bullet)\cong F(\bullet)[\sq]$$ in $\hmf{\ann\times\mon}{\ann}^{/F\circ\pr_1}$. In particular, for a ring $R$ and an $R$-algebra $R'$ there is a canonical isomorphism $$R[\sq]\otimes_RR'\cong R'[\sq]$$ in $\hmf{\mon}{\alg(R')}$. Furthermore, as $R[\sq]$ commutes with coproducts the above yields for $k\in\N$ a canonical isomorphism $$\bullet\bigl[\bigoplus_{i=1}^k\sq_i\bigr]\cong\bigotimes_{i=1}^k\!{}_{{}_{\bullet}}\bullet[\sq_i]\cong\bullet[\sq_1]\cdots[\sq_k]$$ in $\hmf{\ann\times\mon^k}{\ann}^{/\pr_1}$.
\end{no}

\begin{no}\label{alg10}
Let $M$ be a monoid. For $T\subseteq M$ we denote by $M-T$ the monoid of differences of $M$ with negatives in $T$ and by $\eps_T$ the canonical epimorphism $M\rightarrow M-T$; if $T=\{t\}$ then we write $M-t$ and $\eps_t$ instead of $M-T$ and $\eps_T$. If $T'$ is the submonoid of $M$ generated by $T$ then $M-T=M-T'$. If $T=M$ then $\diff(M)\dfgl M-T$ is a group, called the group of differences of $M$. An element $m\in M$ is called \textit{cancellable} if $m+k=m+l$ implies $k=l$ for $k,l\in M$, and $M$ is called \textit{cancellable} if every element of $M$ is cancellable. Furthermore, $M$ is called \textit{torsionfree} if $rm=rn$ implies $m=n$ for $m,n\in M$ and $r\in\N^*$, and \textit{integrally closed} if it is cancellable and $rm\in M$ implies $m\in M$ for $m\in\diff(M)$ and $r\in\N^*$. If $T\subseteq M$, then $\eps_T$ is a monomorphism if and only if every $m\in T$ is cancellable. Hence, $M$ is cancellable if and only if $\eps_M$ is a monomorphism, and then we consider $M$ as a submonoid of $\diff(M)$. If $M$ is cancellable, then it is torsionfree if and only if $\diff(M)$ is so.
\end{no}

\begin{no}\label{alg80}
Let $R$ be a ring, and let $M$ be a monoid. If $T\subseteq M$ is a subset then the $R[M]$-algebras\footnote{For a subset $S\subseteq R$ we denote by $S^{-1}R$ the ring of fractions of $R$ with denominators in $S$ and by $\eta_S$ the canonical morphism $R\rightarrow S^{-1}R$; if $S=\{s\}$ then we write $R_s$ and $\eta_s$ instead of $S^{-1}R$ and $\eta_S$.} $$R[\eps_T]\colon R[M]\rightarrow R[M-T]\;\text{ and }\;\eta_{\exp_{R,M}(T)}\colon R[M]\rightarrow\exp_{R,M}(T)^{-1}(R[M])$$ are canonically isomorphic. So, if $t\in M$ then the $R[M]$-algebras $$R[\eps_t]\colon R[M]\rightarrow R[M-t]\;\text{ and }\;\eta_{e_t}\colon R[M]\rightarrow R[M]_{e_t}$$ are isomorphic, and thus $$\spec(\bullet[\eps_t])\colon\spec(\bullet[M-t])\rightarrow\spec(\bullet[M])$$ is an open immersion\footnote{By abuse of language, a morphism in $\hmf{\alg(R)^{\circ}}{\sch_{/R}}$ is called an (\textit{open} or \textit{closed}) \textit{immersion} if its values are so.} in $\hmf{\ann^{\circ}}{\sch}$ (\cite[I.1.6.6]{ega}).

Now, let $S\subseteq R[M]$ be a subset, and set $T\dfgl S\cap R$ and $U\dfgl\eta_T[M](S)$. By \ref{alg40} there are unique morphisms $f,g,h$ of rings yielding commutative diagrams $$\xymatrix@R16pt@C50pt{R\ar[r]^(.45){\eta_T}\ar[d]_{s(R,M)}&T^{-1}R\ar[r]^(.45){s(T^{-1}R,M)}\ar[d]^f&(T^{-1}R)[M]\ar[d]^{\eta_U}\\R[M]\ar[r]^(.4){\eta_S}&S^{-1}(R[M])&U^{-1}((T^{-1}R)[M])\ar[l]_(.55)g}$$ and $$\xymatrix@R16pt@C50pt{R\ar[r]^(.45){s(R,M)}\ar[d]_{\eta_T}&R[M]\ar[r]^(.45){\eta_S}\ar[d]^{\eta_T[M]}&S^{-1}(R[M])\ar[d]^h\\T^{-1}R\ar[r]^(.41){s(T^{-1}R,M)}&(T^{-1}R)[M]\ar[r]^(.43){\eta_U}&U^{-1}((T^{-1}R)[M])}$$ in $\ann$. As $g$ and $h$ are mutually inverse we see that the $R[M]$-algebras $$\eta_S\colon R[M]\rightarrow S^{-1}(R[M])\;\text{ and }\;\eta_U\circ\eta_T[M]\colon R[M]\rightarrow U^{-1}((T^{-1}R)[M])$$ are canonically isomorphic.
\end{no}

\begin{no}
Let $I$ be a category. On use of the canonical isomorphism $$\hmf{I^{\circ}}{\ann\times\mon}\cong\hmf{I^{\circ}}{\ann}\times\hmf{I^{\circ}}{\mon}$$ we obtain from \ref{alg20} a diagram $$\xymatrix@R5pt@C60pt{&\ar@{=>}@<-4ex>[dd]^{\hmf{I^{\circ}}{s}}&\\\hmf{I^{\circ}}{\ann}\times\hmf{I^{\circ}}{\mon}\ar@/^4ex/[rr]^{\pr_1}\ar@/_4ex/[rr]_{\hmf{I^{\circ}}{\bullet[\lsq]}}&&\hmf{I^{\circ}}{\ann}.\\&&}$$ If no confusion can arise we write $\bullet[\sq]$ instead of $\hmf{I^{\circ}}{\bullet[\sq]}$ and $s$ instead of $\hmf{I^{\circ}}{s}$. Denoting by $c\colon\ann\rightarrow\hmf{I^{\circ}}{\ann}$ the functor that maps a ring $R$ onto the constant functor with value $R$ and setting $c'\dfgl c\times\Id_{\hmf{I^{\circ}}{\mon}}$ we get a diagram $$\xymatrix@R5pt@C60pt{&\ar@{=>}@<-1.5ex>[dd]^{s\circ c'}&\\\ann\times\hmf{I^{\circ}}{\mon}\ar@/^4ex/[rr]^{c\circ\pr_1}\ar@/_4ex/[rr]_{\bullet[\lsq]\circ c'}&&\hmf{I^{\circ}}{\ann}.\\&&}$$ Denoting by $d\colon\sch\rightarrow\hmf{I}{\sch}$ the functor that maps a scheme $S$ onto the constant functor with value $S$ we end up with a diagram $$\xymatrix@R5pt@C60pt{&&\\\ann^{\circ}\times\hmf{I^{\circ}}{\mon}^{\circ}\ar@/^4ex/[rr]^{d\circ\spec\circ\pr_1}\ar@/_4ex/[rr]_{\spec(\bullet[\lsq])\circ c'}&&\hmf{I}{\sch}.\\&\ar@{=>}@<1.5ex>[uu]_{\hmf{I^{\circ}}{t}\circ c'}&}$$  If no confusion can arise we write $t$ instead of $\hmf{I^{\circ}}{t}$.
\end{no}

\begin{no}\label{alg130}
Let $I$ be a preordered set and let $\M=((M_i)_{i\in I},(p_{ij})_{i\leq j})$ be a projective system in $\mon$ over $I$. For $i\in I$ we set $$X_{\M,i}(\bullet)\dfgl\spec(\bullet[M_i])\colon\ann^{\circ}\rightarrow\sch$$ and $$t_{\M,i}(\bullet)\dfgl t(\bullet,M_i)\colon X_{\M,i}(\bullet)\rightarrow\spec(\bullet),$$ and for $i,j\in I$ with $i\leq j$ we set $$\iota_{\M,i,j}(\bullet)\dfgl\spec(\bullet[p_{ij}])\colon X_{\M,i}(\bullet)\rightarrow X_{\M,j}(\bullet).$$ If $R$ is a ring then we say that $\M$ is \textit{openly immersive for $R$} if $\iota_{\M,i,j}(R)$ is an open immersion for all $i,j\in I$ with $i\leq j$, and then it is openly immersive for every $R$-algebra (\ref{alg40}, \cite[I.4.3.6]{ega}). We say that $\M$ is \textit{openly immersive} if it is openly immersive for $\Z$, or -- equivalently -- for every ring.

Now, let $R$ be a ring, and suppose that $I$ is a lower semilattice and that $\M$ is openly immersive for $R$. Then, \cite[I.2.4.1; 0.4.1.7]{ega} implies that the inductive system $$\bigl(\bigl(X_{\M,i}(R)\xrightarrow{t_{\M,i}(R)}\spec(R)\bigr)_{i\in I},(\iota_{\M,i,j}(R))_{i\leq j}\bigr)$$ in $\sch_{/R}$ has an inductive limit $$\tm(R)\colon\xm(R)\rightarrow\spec(R)$$ such that for every $i\in I$ the canonical morphism $\iota_{\M,i}(R)\colon X_{\M,i}(R)\rightarrow \xm(R)$ is an open immersion by means of which we consider $X_{\M,i}(R)$ as an open subscheme of $\xm(R)$. The inductive limit $\xm(R)$ can be understood as obtained by gluing the family $(X_{\M,i}(R))_{i\in I}$ along $(X_{\M,\inf(i,j)}(R))_{(i,j)\in I^2}$. In particular, $(X_{\M,i}(R))_{i\in I}$ is an affine open covering of $\xm(R)$ whose image is closed under nonempty, finite intersections.

The above gives rise to a functor $$\xm\colon\alg(R)^{\circ}\rightarrow\sch_{/R}$$ over $\spec$ together with open immersions $\iota_{\M,i}\colon X_{\M,i}\rightarrow \xm$. For an $R$-algebra $R'$ the diagram of categories $$\xymatrix@R15pt{\alg(R')\ar[r]^{\xm}&\sch_{/R'}\\\alg(R)\ar[r]^{\xm}\ar[u]^{\bullet\otimes_RR'}&\sch_{/R}\ar[u]_{\bullet\times_RR'}}$$ commutes up to canonical isomorphism (\ref{alg40}). So, defining a functor $$\xm\colon\sch\rightarrow\sch$$ over $\Id_{\sch}$ (with structural morphism denoted by $\tm$) by setting $$\xm(\bullet)\dfgl \xm(\Z)\times\bullet,$$ it follows that $\xm(\spec(R))$ and $\xm(R)$ for a ring $R$ are canonically isomorphic and will henceforth be identified, and that we have an open covering $(X_{\M,i})_{i\in I}$ of $\xm$. If $S$ is a scheme and $(U_i)_{i\in I}$ is an open covering of $S$, then the canonical injections $U_i\hookrightarrow S$ induce open immersions $\xm(U_i)\rightarrowtail \xm(S)$ by means of which we consider the $\xm(U_i)$ as open subschemes of $\xm(S)$, and it follows that $(\xm(U_i))_{i\in I}$ is an open covering of $\xm(S)$.

If $S$ is a scheme and $i\in I$ then the codiagonal of the free $\Z$-module $\Z[M_i]$ is a retraction of the canonical injection $\Z\hookrightarrow\Z[M_i]$ and in particular a surjective morphism in $\ann$, hence it induces by base change a closed immersion $S\rightarrowtail X_{\M,i}(S)$ that is a section of $t_{\M,i}(S)$. So, if $I\neq\emptyset$ we get an immersion $S\rightarrowtail \xm(S)$ that is a section of $\tm(S)$, called canonical.
\end{no}

\begin{exa}\label{alg150}	
If $I$ is a lower semilattice and $\M=((M_i)_{i\in I},(p_{ij})_{i\leq j})$ is a projective system in $\mon$ over $I$ such that for $i,j\in I$ with $i\leq j$ there is a $t_{ij}\in M_j$ with $M_i=M_j-t_{ij}$ and $p_{ij}=\eps_{t_{ij}}$, then $\M$ is openly immersive (\ref{alg80}) and hence gives rise to a functor $\xm\colon\sch\rightarrow\sch$ over $\Id_{\sch}$.
\end{exa}

\begin{no}\label{f1}
The above has some connections with the ``geometry over the field with one element''. There, one is mainly concerned with what happens ``below $\Z$'', while we want to know what happens ``above $\Z$'', i.e., after changing the base to an arbitrary ring. This question has seemingly not yet been studied much, although a few related observations for $\fun$-schemes can be found in \cite{chww} and \cite{deitmar2}.

We try now to shed light upon the aforementioned connections. There are different approaches to $\fun$ -- see \cite{map} for an overview. The one that seems closest to our theory is Deitmar's (\cite{deitmar1}), which is equivalent to the approach by To\"en-Vaqui\'e (\cite{tv}, \cite{vezzani}). In \cite{deitmar1}, an $\fun$-scheme is defined as a topological space furnished with a sheaf of monoids that is locally isomorphic to spectra of monoids.\footnote{In more recent works on Deitmar's approach as well as in the more general approaches by Connes and Consani (\cite{chu}, \cite{cc1}, \cite{cc3}, \cite{cc2}), monoids are replaced by pointed monoids, but since ascent to $\Z$ includes factoring out the absorbing element this difference need not bother us.}

Let $\mathbbm{M}=((M_i)_{i\in I},(p_{ij})_{i\leq j})$ be an openly immersive projective system in $\mon$ over a lower semilattice $I$. Taking spectra this yields an inductive system $$\bigl((\spec(M_i))_{i\in I},(\spec(p_{ij}))_{i\leq j}\bigr)$$ of (affine) $\fun$-schemes. If the morphisms $\spec(p_{ij})$ are open immersions, then this inductive system has an inductive limit $Z_{\mathbbm{M}}$ that can be understood as the $\fun$-scheme obtained by gluing $(\spec(M_i))_{i\in I}$ along $(\spec(M_{\inf(i,j)}))_{i,j\in I^2}$, and ascent to $\Z$ maps this $\fun$-scheme to $X_{\mathbbm{M}}(\Z)$, i.e., $$Z_{\mathbbm{M}}\otimes_{\fun}\Z\cong X_{\mathbbm{M}}(\Z)$$ (\cite{deitmar1}, cf.~\cite[Proposition 42]{vezzani}). So, we need to know whether or not the morphisms $\spec(p_{ij})$ are open immersions of $\fun$-schemes. By \cite[Theorem 30]{vezzani} this holds if and only if $p_{ij}$ is isomorphic to the canonical morphism of a monoid of differences obtained by inverting a single element. Therefore, we can construct $Z_{\mathbbm{M}}$ as above if and only if $\mathbbm{M}$ is of the form described in \ref{alg150}.

Conversely, we may ask whether every $\fun$-scheme is of the form $Z_{\mathbbm{M}}$ for an openly immersive projective system $\mathbbm{M}$ as in \ref{alg150}. At present the author is ignorant of an answer to this. There is an obvious notion of quasiseparatedness for $\fun$-schemes such that the analogue of \cite[I.6.1.12]{ega} holds, and $\fun$-schemes of the form $Z_{\mathbbm{M}}$ are clearly quasiseparated. So, it would be interesting to know whether $\fun$-schemes are always quasiseparated. (Note that they are not necessarily separated by \cite[3.5]{chww}.)
\end{no}


\section{Geometric properties}\label{sec3}

\noindent\textit{Throughout this section let $I$ be a lower semilattice, let\/ $\M=((M_i)_{i\in I},(p_{ij})_{i\leq j})$ be an openly immersive projective system in\/ $\mon$ over $I$, let $r\dfgl\sup_{i\in I}(\rk(\diff(M_i)))$, and let $S$ be a scheme. By abuse of language, we call\/ $\M$ cancellable, torsionfree, integrally closed, of finite type, finite, or zero if $M_i$ is so for every $i\in I$.}\smallskip

We start by treating separation, flatness and finiteness properties.

\begin{prop}\label{n10}
a) $\tm(S)$ is quasiseparated; it is separated if and only if\linebreak $M_{\inf(i,j)}$ is generated by $p_{\inf(i,j)i}(M_i)\cup p_{\inf(i,j)j}(M_j)$ for $i,j\in I$ or $S=\emptyset$.

b) If $I$ is finite then $\tm(S)$ is quasicompact.

c) $\tm(S)$ is flat; it is faithfully flat if and only if $I\neq\emptyset$ or $S=\emptyset$.
\end{prop}

\begin{proof}
Considering its covering $(X_{\M,i}(\Z))_{i\in I}$ we see that $\xm(\Z)$ is quasiseparated and quasicompact, yielding the first claim in a) and b) by base change (\cite[I.6.1.12; I.6.1.10; I.6.1.9; I.6.1.2; I.6.1.5]{ega}). The second claim in a) is clear if $S$ is affine, hence holds in general by base change (\cite[I.5.3.9; I.5.3.1]{ega}). If $i\in I$ then the $\Z$-module underlying $\Z[M_i]$ is free and thus faithfully flat (\cite[I.3.1 Exemple 2]{ac}), hence $t_{\M,i}(\Z)$ is faithfully flat (\cite[IV.2.2.3]{ega}). Therefore, $\tm(\Z)$ is flat (and faithfully flat if $I\neq\emptyset$), hence c) follows by base change (\cite[IV.2.1.4; IV.2.2.13]{ega}).
\end{proof}

\begin{cor}\label{n130}
a) $\xm(S)$ is quasiseparated if and only if $S$ is quasiseparated or $I=\emptyset$.

b) If $\xm(S)$ is separated then $S$ is separated or $I=\emptyset$; if $M_{\inf(i,j)}$ is generated by $p_{\inf(i,j)i}(M_i)\cup p_{\inf(i,j)j}(M_j)$ for $i,j\in I$ or $S=\emptyset$ then the converse holds.

c) If $\xm(S)$ is quasicompact then $S$ is quasicompact or $I=\emptyset$; if $I$ is finite then the converse holds.
\end{cor}

\begin{proof}
Considering the canonical immersion $S\rightarrowtail\xm(S)$ we see that (quasi-)sepa\-ratedness of $\xm(S)$ implies (quasi-)separatedness of $S$, while the converse statements follow from \ref{n10} a) (\cite[I.6.1.9; I.5.3.1]{ega}). Quasicompactness of $\xm(S)$ implies quasicompactness of $S$ by \ref{n10} c), while the converse follows from \ref{n10} b) (\cite[I.6.1.6; I.6.1.5]{ega}).
\end{proof}

\begin{prop}\label{n60}
Consider the following statements: (1)\/ $\M$ is of finite type or $S=\emptyset$; (2) $\tm(S)$ is locally of finite presentation; (3) $\tm(S)$ is locally of finite type; (4) $\tm(S)$ is of finite presentation; (5) $\tm(S)$ is of finite type. Then, (1)--(3) are equivalent; if $I$ is finite then (1)--(5) are equivalent.
\end{prop}

\begin{proof}
If $\M$ is of finite type then $t_{\M,i}(\Z)$ is of finite presentation for every $i\in I$, hence $\tm(\Z)$ is locally of finite presentation, and thus $\tm(S)$ is locally of finite presentation (\cite[I.6.2.5; I.6.2.1.2; I.6.2.6]{ega}); so, (1) implies (2). If $\tm(S)$ is locally of finite type and $S\neq\emptyset$ then there exists a field $K$ such that $\tm(K)$ is locally of finite type, hence if $i\in I$ then the $K$-algebra $K[M_i]$ is of finite type, and thus $M_i$ is of finite type (\cite[I.6.2.3; I.6.2.5]{ega}, \cite[7.7]{gilmer}); so, (3) implies (1). The claim follows now from \ref{n10} a), b).
\end{proof}

Next we give conditions for several properties related to noetherianness and the property of being jacobsonian to be respected and reflected by $\xm$. We call a scheme topologically (locally) noetherian if its underlying topological space is (locally) noetherian, and pointwise noetherian if the stalks of its structure sheaf are noetherian.

\begin{prop}\label{n70}
a) $\xm(S)$ is locally noetherian if and only if $S$ is locally noetherian and\/ $\M$ is of finite type, or $I=\emptyset$.

b) If $\xm(S)$ is noetherian then $S$ is noetherian and\/ $\M$ is of finite type, or $I=\emptyset$; if $I$ is finite then the converse holds.

c) If $\xm(S)$ is pointwise noetherian then $S$ is pointwise noetherian or $I=\emptyset$; if $\M$ is of finite type then the converse holds.

d) Consider the following statements: (1) $\xm(S)$ is topologically locally noetherian; (2) $S$ is topologically locally noetherian or $I=\emptyset$. If $I$ is finite then it holds (1)$\Rightarrow$(2); if\/ $\M$ is of finite type then it holds (2)$\Rightarrow$(1).

e) If $\xm(S)$ is topologically noetherian then $S$ is topologically noetherian or $I=\emptyset$; if $I$ is finite and\/ $\M$ is of finite type then the converse holds.

f) If $\xm(S)$ is a jacobsonian then $S$ is jacobsonian or $I=\emptyset$; if\/ $\M$ is of finite type then the converse holds.
\end{prop}

\begin{proof}
Suppose that $\xm(S)$ is locally noetherian and $I\neq\emptyset$. If $U=\spec(R)$ is an affine open subscheme of $S$ and $i\in I$ then $X_{\M,i}(U)$ is an open subscheme of $\xm(S)$, hence locally noetherian, thus $R[M_i]$ is noetherian, and therefore $U$ is noetherian and $M_i$ is of finite type (\cite[I.4.1.7; I.2.7.3]{ega}, \cite[7.7]{gilmer}). So, $S$ is locally noetherian and $\M$ is of finite type. The converse is clear by \ref{n60} and \cite[I.6.2.2]{ega}, and so a) is proven. b) is clear by \cite[I.2.7.2]{ega}, a) and \ref{n130} c). Considering the canonical immersion $S\rightarrowtail\xm(S)$ we see that if $\xm(S)$ is pointwise noetherian then $S$ is so or $I=\emptyset$. Conversely, suppose that $S$ is pointwise noetherian. If $U=\spec(R)$ is an affine open subscheme of $S$, $i\in I$, and $\mathfrak{p}\in\spec(R[M_i])$, then $R_{\mathfrak{p}\cap R}$ is noetherian, hence so is the ring of fractions $R[M_i]_{\mathfrak{p}}$ of $R_{\mathfrak{p}\cap R}[M_i]$, and therefore $X_{\M,i}(R)$ is pointwise noetherian. Thus, $\xm(S)$ is pointwise noetherian, and c) is proven. To show the first claims in d) and e) it suffices by \ref{n10} b), c) to show that the target of a surjective (and quasicompact) morphism of schemes whose source is topologically (locally) noetherian is also topologically (locally) noetherian; this is straightforward to prove. The second claims in d) and e) follow from \cite[2.6]{ohm-pen} or from d) and \ref{n130} c), respectively. Finally, considering the canonical immersion $S\rightarrowtail\xm(S)$, the first claim in f) holds by \cite[I.6.4.3]{ega}, while the second one follows from \ref{n60} and \cite[I.6.4.7]{ega}.
\end{proof}

Now we turn to catenarity and universal catenarity. We call a ring $R$ universally catenary if every $R$-algebra of finite type is catenary, and we call a scheme universally catenary if the stalks of its structure sheaf are universally catenary.\footnote{In contrast to \cite[IV.5.6.3]{ega} we do not require a universally catenary scheme to be locally noetherian.}

\begin{prop}
We consider the following statements: (1) $S$ is universally catenary or $I=\emptyset$; (2) $\xm(S)$ is universally catenary; (3) $\xm(S)$ is catenary. If $\M$ is of finite type then it holds (1)$\Leftrightarrow$(2)$\Rightarrow$(3); if\/ $\M$ is cancellable, torsionfree, of finite type, and nonzero and $S$ is pointwise noetherian, then (1)--(3) are equivalent.
\end{prop}

\begin{proof}
(1) implies (2) by \ref{n60}. Considering the canonical immersion $S\rightarrowtail\xm(S)$ we see that (2) implies (1). Suppose that the conditions of the second statement are fulfilled and that $\xm(S)$ is catenary. Let $U=\spec(R)$ be an affine open subscheme of $S$, so that $\xm(U)$ is an open subscheme of $\xm(S)$ and hence catenary. There exists $i\in I$ with $M_i\neq 0$, implying $\diff(M_i)\cong\Z^m$ for $m\in\N^*$. Therefore, $R[\Z^m]$ is a ring of fractions of the catenary ring $R[M_i]$, hence catenary, and $R[\Z]$ is a quotient of the catenary ring $R[\Z^m]$, hence catenary. Writing $R[T]\dfgl R[\N]$, the open sets $D(T)$ and $D(T-1)$ cover $\spec(R[T])$ and are both isomorphic to $\spec(R[\Z])$ and thus catenary. Therefore, $R[\N]$ is catenary, and thus $R$ is universally catenary by Ratliff's Theorem \cite[2.6]{ratliff} (cf.~\cite[Theorem 31.7, Corollary 1]{mat}) which is easily seen to hold for pointwise noetherian instead of noetherian rings. So, $S$ is universally catenary.\footnote{The proof that catenarity of $R[\Z]$ implies catenarity of $R[\N]$ is due to Laurent Moret-Bailly via {\tt MathOverflow.net}.}
\end{proof}

The following results give conditions under which $\tm(S)$ is irreducible, connected, reduced, or normal -- see \cite[IV.4.5.5; IV.6.8.1]{ega} for the relevant definitions --, and under which $\xm$ respects and reflects irreducibility, connectivity, reducedness, and normality. As applications we show that $\xm$ commutes with $\bullet_{\red}$, and we give a description of the irreducible or connected components of $\xm(S)$ in terms of the corresponding components of $S$.

\begin{prop}\label{n90}
Suppose that $\M$ is torsionfree. If\/ $\M$ is cancellable then $\tm(S)$ is connected; if moreover $I\neq\emptyset$ then $\tm(S)$ is irreducible. If\/ $\M$ is of finite type then $\tm(S)$ is reduced; if\/ $\M$ is moreover integrally closed then $\tm(S)$ is normal.
\end{prop}

\begin{proof}
Let $K$ be a field and let $i\in I$. If $\M$ is cancellable then $X_{\M,i}(K)$ is integral (\cite[8.1]{gilmer}), hence irreducible, and in particular nonempty and connected. Thus, if $I\neq\emptyset$ then considering its covering $(X_{\M,i}(K))_{i\in I}$ it follows from \cite[II.4.1 Proposition 7]{ac} that $\xm(K)$ is irreducible. This implies the first claim. If $\M$ is of finite type then $X_{\M,i}(K)$ is reduced (\cite[9.9]{gilmer}), and if $\M$ is moreover integrally closed then $X_{\M,i}(K)$ is normal (\cite[12.6]{gilmer}, \cite[V.1.5 Proposition 16, Corollaire 1]{ac}). So, $\xm(K)$ is reduced (or normal, respectively), and thus the second claim follows from \ref{n10} c) and \ref{n70} a).
\end{proof}

\begin{prop}\label{n240}
If $\xm(S)$ is reduced then $S$ is reduced or $I=\emptyset$; if\/ $\M$ is torsionfree then the converse holds.
\end{prop}

\begin{proof}
If $\xm(S)$ is reduced and $I\neq\emptyset$ then $S$ is reduced by \ref{n10} c) and \cite[IV.2.1.13]{ega}. The converse is clear if $I=\emptyset$. If $S$ is reduced, $U$ is an affine open subscheme of $S$, and $i\in I$, then $U$ is reduced, hence so is $X_{\M,i}(U)$ (\cite[9.9]{gilmer}), and therefore $\xm(S)$ is reduced.
\end{proof}

\begin{cor}\label{n245}
If\/ $\M$ is torsionfree then there is a canonical isomorphism $$\xm(\bullet)_{\red}\cong\xm(\bullet_{\red}).$$
\end{cor}

\begin{proof}
Let $S$ be a scheme. We have a commutative diagram of schemes $$\xymatrix@C50pt@R15pt{&\xm(S_{\red})\ar[r]^{\tm(S_{\red})}\ar[d]&S_{\red}\ar[d]\\\xm(S)_{\red}\ar[r]^i\ar[ru]^j&\xm(S)\ar[r]^{\tm(S)}&S,}$$ where the unmarked morphisms are the canonical ones and the quadrangle is cartesian (\cite[I.4.5.13]{ega}). Moreover, $i$ is a closed immersion, hence $j$ is a monomorphism, the vertical morphisms are homeomorphisms (\cite[I.4.5.3]{ega}), and $\xm(S_{\red})$ is reduced (\ref{n240}). Therefore, $\xm(S_{\red})$ is a reduced subscheme of $\xm(S)$ with underlying topological space the underlying topological space of $\xm(S)$, and thus equal to $\xm(S)_{\red}$. So, $j$ is an isomorphism that is moreover functorial in $S$.
\end{proof}

\begin{prop}\label{n220}
a) If $\xm(S)$ is connected then $S$ is connected or $I=\emptyset$; if\/ $\M$ is torsionfree and cancellable, and if moreover $I$ is finite, or\/ $\M$ is of finite type, or $S$ is affine and $I$ has a smallest element, then the converse holds.

b) If $\xm(S)$ is irreducible then $S$ is irreducible and $I\neq\emptyset$; if\/ $\M$ is torsionfree and cancellable then the converse holds.

c) If $\xm(S)$ is integral then $S$ is integral and $I\neq\emptyset$; if\/ $\M$ is torsionfree and cancellable then the converse holds.

d) If $\xm(S)$ is normal then $S$ is normal or $I=\emptyset$; if\/ $\M$ is torsionfree and integrally closed then the converse holds.
\end{prop}

\begin{proof}
a) The first claim holds by \ref{n10} c) and \cite[I.11.2 Proposition 4]{tg}. Suppose that the conditions of the second statement are fulfilled and that $S$ is connected and $I\neq\emptyset$. If $I$ is finite or $\M$ is of finite type then \ref{n10} b), c), \ref{n60}, \ref{n90} and \cite[I.7.3.10; IV.4.5.7]{ega} imply that $\xm(S)$ is connected. If $S$ is affine then $X_{\M,i}(S)$ is connected (\cite[11.7; 10.8]{gilmer}, \cite[II.4.3 Proposition 15, Corollaire 2]{ac}), and since $I$ has a smallest element it follows that $\xm(S)$ is connected (\cite[I.11.1 Proposition 2]{tg}). b) The first claim holds by \ref{n10} c) and \cite[II.4.1 Proposition 4]{ac}. For the converse it follows from \cite[II.4.1 Proposition 3]{ac} that it suffices to show that if $R$ is a ring such that $\spec(R)$ is irreducible and $i\in I$, then $X_{\M,i}(R)$ is irreducible. But in this situation, $R_{\red}$ is integral, hence $R[M_i]_{\red}\cong R_{\red}[M_i]$ is integral (\ref{n245}, \cite[8.1]{gilmer}) and thus $X_{\M,i}(R)$ is irreducible. c) follows immediately from b) and \ref{n240}. d) The first claim holds by \ref{n10} c) and \cite[IV.2.1.13]{ega}. For the converse, suppose that $S$ is normal, let $U=\spec(R)$ be an affine open subscheme of $S$, let $i\in I$, and let $\mathfrak{p}\in\spec(R[M_i])$. Then, $R$ is normal, hence $R_{\mathfrak{p}\cap R}[M_i]$ is integral and integrally closed (\cite[12.11]{gilmer}), and thus so is its ring of fractions $R[M_i]_{\mathfrak{p}}$ (\cite[V.1.5 Proposition 16, Corollaire 1]{ac}). Therefore, $X_{\M,i}(R)$ is normal, and the claim is proven.
\end{proof}

\begin{cor}\label{n270}
Suppose that\/ $\M$ is torsionfree and cancellable, and that $I\neq\emptyset$.

a) The maps $Z\mapsto\xm(Z)$ and $Z\mapsto\tm(S)(Z)$ induce mutually inverse bijections between the sets of irreducible components of $S$ and $\xm(S)$.

b) If $I$ is finite, or\/ $\M$ is of finite type, or $S$ is affine and $I$ has a smallest element, then the maps $Z\mapsto\xm(Z)$ and $Z\mapsto\tm(S)(Z)$ induce mutually inverse bijections between the sets of connected components of $S$ and $\xm(S)$.
\end{cor}

\begin{proof}
a) If $Z$ is an irreducible component of $S$ then $\xm(Z)$ is an irreducible closed subset of $\xm(S)$ (\ref{n220} b)), and if $Y$ is an irreducible closed subset of $\xm(S)$ containing $\xm(Z)$ then $\tm(S)(Y)$ is irreducible and contains $Z$, implying $\tm(S)(Y)=Z$ and thus $\xm(Z)=Y$. Conversely, if $Z$ is an irreducible component of $\xm(S)$ then $\tm(S)(Z)$ is irreducible, and if $Y$ is an irreducible closed subset of $S$ containing $\tm(S)(Z)$ then \ref{n220} b) shows that $\xm(Y)$ is an irreducible closed subset of $\xm(S)$ containing $Z$, implying $Z=\xm(Y)$. This proves the claim. b) is shown analogously on use of \ref{n220} a). 
\end{proof}

Next we treat Serre's properties $(S_k)$ and Cohen-Macaulayness. Recall that for $k\in\N$ a scheme $X$ is said to have property $(S_k)$ if it is locally noetherian and  $\depth(\mathscr{O}_{X,x})\geq\min\{k,\dim(\mathscr{O}_{X,x})\}$ for every $x\in X$, and a morphism of schemes is said to have property $(S_k)$ if it is flat and its fibres have $(S_k)$. Furthermore, a scheme or morphism of schemes is called Cohen-Macaulay if it has $(S_k)$ for every $k\in\N$ (\cite[IV.5.7.1--2; IV.6.8.1]{ega}).

\begin{prop}\label{n80}
If\/ $\M$ is torsionfree, of finite type, and integrally closed, then $\tm(S)$ is a Cohen-Macaulay morphism.
\end{prop}

\begin{proof}
If $K$ is a field then $\xm(K)$ is a Cohen-Macaulay scheme by \ref{n70} a) and Hochster's Theorem \cite[Theorem 1]{hochster} (cf. \cite[6.10]{bg}), so the claim follows from \ref{n10} c).
\end{proof}

\begin{cor}
Suppose that\/ $\M$ is torsionfree, of finite type, and integrally closed.

a) If $k\in\N$, then $\xm(S)$ has $(S_k)$ if and only if $S$ has $(S_k)$ or $I=\emptyset$.

b) $\xm(S)$ is a Cohen-Macaulay scheme if and only if $S$ is a Cohen-Macaulay scheme or $I=\emptyset$.
\end{cor}

\begin{proof}
Immediately from \ref{n70} a), \ref{n10} c), \ref{n80} and \cite[IV.6.4.1]{ega}.
\end{proof}

The last task in this section is to describe $\dim(\xm(S))$ in terms of $\dim(S)$. Under certain hypotheses we give lower and upper bounds, and a precise formula if $S$ is in addition locally noetherian. This allows us to quickly derive characterisations of equidimensionality, artinianness, and finiteness of schemes of the form $\xm(S)$.

\begin{prop}\label{n310}
If\/ $\M$ is cancellable and $r\in\N$ then $$\dim(S)+r\leq\dim(\xm(S))\leq(r+1)\dim(S)+r;$$ if $S$ is moreover locally noetherian then $$\dim(S)+r=\dim(\xm(S)).$$
\end{prop}

\begin{proof}
If $R$ is a ring and $M$ is a cancellable monoid with $m\dfgl\rk(\diff(M))\in\N$, then $\dim(R[M])=\dim(R[\diff(M)])=\dim(R[\N^{\oplus m}])$ (\cite[21.4; 17.1]{gilmer}), hence $\dim(R)+m\leq\dim(R[M])\leq(m+1)\dim(R)+m$ (\cite[4.1]{arnold}), and if $R$ is moreover noetherian then $\dim(R)+m=\dim(R[M])$ (\cite[VIII.3.4 Proposition 7, Corollaire 3]{ac}). Choosing an affine open covering of $S$, this yields the claim (\cite[IV.5.1.4]{ega}).
\end{proof}

\begin{cor}
a) If\/ $\M$ is torsionfree and cancellable, $r<\infty$, and $S$ is locally noetherian, then $\xm(S)$ is equidimensional if and only if $S$ is equidimensional or $I=\emptyset$.

b) If\/ $\M$ is cancellable and of finite type, and $I$ is finite, then $\xm(S)$ is artinian if and only if $S$ is artinian and\/ $\M$ is finite, or $S=\emptyset$, or $I=\emptyset$.

c) If\/ $\M$ is torsionfree, cancellable, and of finite type, and $I$ is finite, then $\tm(S)$ is finite if and only if\/ $\M$ is zero or $S=\emptyset$.
\end{cor}

\begin{proof}
a) If $I\neq\emptyset$ then $S$ is equidimensional if and only if $\dim(Z)=\dim(S)$ for every irreducible component $Z$ of $S$. Since $S$ is locally noetherian this holds if and only if $\dim(\xm(Z))=\dim(Z)+r=\dim(S)+r$ for every irreducible component $Z$ of $S$ (\ref{n310}), and this is equivalent to $\xm(S)$ being equidimensional (\ref{n270}). b) If $S=\emptyset$ or $I=\emptyset$ this is clear. Otherwise, $\xm(S)$ is artinian if and only if $\xm(S)$ is noetherian and $\dim(\xm(S))\leq 0$. This holds if and only if $S$ is noetherian and $\dim(S)+r\leq 0$ (\ref{n70} b), \ref{n310}), hence if and only if $S$ is noetherian, $\dim(S)\leq 0$, and $r=0$, and thus if and only if $S$ is artinian and $r=0$ (\cite[I.2.8.2]{ega}). c) Suppose that $\tm(S)$ is finite and that $S\neq\emptyset$ and $I\neq\emptyset$. There exists a field $K$ such that $\tm(K)$ is finite (\cite[II.6.1.5]{ega}). It follows that $\xm(K)$ is locally noetherian and $\tm(K)$ is faithfully flat  (\ref{n70} a), \ref{n10} c)), implying $r=\dim(K)+r=\dim(\xm(R))=\dim(K)=0$ (\ref{n310}, \cite[IV.5.4.2]{ega}), and therefore $M_i=0$ for every $i\in I$. The converse is clear.
\end{proof}


\section{Toric schemes}\label{sec5}

\noindent\textit{Throughout this section let $V$ be an $\R$-vector space of finite dimension, let $n\dfgl\dim_{\R}(V)$, let $N$ be a\/ $\Z$-structure on $V$, and let $M\dfgl N^*$ denote the dual\/ $\Z$-structure on the dual space $V^*$.}\footnote{A \textit{$\Z$-structure on $V$} is a subgroup $N$ of the additive group underlying $V$ with $\langle N\rangle_{\R}=V$ and $\rk_{\Z}(N)=n$, or -- equivalently -- such that the image of the canonical morphism of $\Q$-vector spaces $\Q\otimes_{\Z}N\rightarrow V$ with $a\otimes x\mapsto ax$ is a $\Q$-structure on $V$ in the sense of \cite[II.8.1]{a}. Note that $M$ is canonically isomorphic to and identified with a $\Z$-structure on $V^*$, and that the canonical identification of $V$ and $V^{**}$ identifies $N$ with $M^*$.}\smallskip

Before defining toric schemes we briefly recall some terminology and facts about polycones and fans and refer the reader to \cite{bg} for details and proofs.

\begin{no}
For $A\subseteq V$ we denote by $A^{\vee}$ the dual $\{u\in V^*\mid u(A)\subseteq\R_{\geq 0}\}$ of $A$, we set $A^{\vee}_M\dfgl A^{\vee}\cap M$, and we denote by $\cone(A)$ the set of $\R$-linear combinations of $A$ with coefficients in $\R_{\geq 0}$. A subset $\sigma\subseteq V$ is called an \textit{$N$-polycone (in $V$)} if it is the intersection of finitely many closed halfspaces in $V$ defined by hyperplanes having bases in $N$ (or -- equivalently -- if there is a finite subset $A\subseteq N$ with $\sigma=\cone(A)$) that does not contain a line. If $\sigma$ is a $N$-polycone then a \textit{face of $\sigma$} is a subset $\tau\subseteq\sigma$ such that there exists $u\in\sigma^{\vee}_M$ with $\tau=\sigma\cap\ke(u)$, and the set $\face(\sigma)$ of faces of $\sigma$ consists of finitely many $N$-polycones. The relation ``$\tau$ is a face of $\sigma$'' is an ordering on the set of $N$-polycones, denoted by $\tau\fleq\sigma$.

An \textit{$N$-fan (in $V$)} is a finite set $\sig$ of $N$-polycones with $\sigma\cap\tau\in\face(\sigma)\subseteq\sig$ for $\sigma,\tau\in\sig$. An $N$-fan $\sig$ is considered as an ordered set by means of the ordering induced by $\sigma\fleq\tau$ (coinciding with the ordering induced by $\sigma\subseteq\tau$), hence it is a finite lower semilattice with $\inf(\sigma,\tau)=\sigma\cap\tau$ for $\sigma,\tau\in\sig$. It is called \textit{full (in $V$)} if $\langle\bigcup\sig\rangle_{\R}=V$, \textit{complete (in $V$)} if $\bigcup\sig=V$, and \textit{$N$-regular} if every $\sigma\in\sig$ is generated by a subset of a $\Z$-basis of $N$.
\end{no}

\begin{no}
Let $\sig$ be an $N$-fan. If $\sigma\in\sig$ then $\sigma^{\vee}_M$ is a torsionfree, integrally closed monoid of finite type (\ref{alg10}). For $\sigma,\tau\in\sig$ it holds $(\sigma\cap\tau)^{\vee}_M=\sigma^{\vee}_M+\tau^{\vee}_M$, and if moreover $\tau\fleq\sigma$ then $\sigma^{\vee}_M$ is a submonoid of $\tau^{\vee}_M$ and there exists $u\in\sigma^{\vee}_M$ such that $\tau^{\vee}_M=\sigma^{\vee}_M-u$ is the monoid of differences obtained from $\sigma^{\vee}_M$ by inverting $u$. Hence, the family $(\sigma^{\vee}_M)_{\sigma\in\sig}$ together with the canonical injections is an openly immersive projective system of submonoids of $M$ over the lower semilattice $\sig$ (\ref{alg130}, \ref{alg150}). By \ref{alg130} it defines a functor $\xs\colon\sch\rightarrow\sch$ over $\Id_{\sch}$ (whose structural morphism we denote by $\ts$), furnished with a finite open covering $(X_{\sigma})_{\sigma\in\sig}$. If $S$ is a scheme then $\ts(S)\colon\xs(S)\rightarrow S$ is called \textit{the toric scheme over $S$ associated with $\sig$.}
\end{no}

When studying a toric scheme (or variety) it may be convenient when the defining fan is full. There is the following easy (but unfortunately noncanonical) reduction to this case, in which toric schemes naturally appear (and which is not available in this form for toric varieties).

\begin{no}
Let $\sig$ be an $N$-fan, let $V'\dfgl\langle\bigcup\sig\rangle_{\R}$, let $N'\dfgl N\cap V'$, and let $\sig'$ denote the set $\sig$ considered as an $N'$-fan in $V'$. Then, there is a (noncanonical) isomorphism of functors $\xs(\bullet)\cong X_{\sig'}(\bullet[N/N'])$. Indeed, it suffices to show $$\xs(\Z)\cong X_{\sig'}(\Z)\otimes\Z[N/N'].$$ We set $M'\dfgl(N')^*$, write $\sigma'$ for $\sigma\in\sig$ considered as an $N'$-polycone in $V'$, and choose an isomorphism $M'\oplus N/N'\overset{\cong}\longrightarrow M$. This induces isomorphisms $$(\sigma')^{\vee}_{M'}\oplus N/N'\overset{\cong}\longrightarrow\sigma^{\vee}_M$$ compatible with the facial relation, hence isomorphisms $$X_{\sigma}(\Z)\overset{\cong}\longrightarrow X_{\sigma'}(\Z)\otimes\Z[N/N']$$ compatible with the facial relation, and thus gluing yields the claim.
\end{no}

Now we can apply the results from Section 2 to derive the basic geometric properties of toric schemes.

\begin{thm}\label{thm1}
a) $\ts(S)$ is a separated, connected, normal Cohen-Macaulay morphism of finite presentation, and if $\sig\neq\emptyset$ then it is irreducible; it is faithfully flat if and only if $\sig\neq\emptyset$ or $S=\emptyset$, and it is finite if and only if $n=0$ or $\sig=\emptyset$ or $S=\emptyset$.

b) $\xs(S)$ is quasiseparated, separated, quasicompact, (locally) noetherian, pointwise noetherian, topologically (locally) noetherian, jacobsonian, connected, reduced, normal, $(S_k)$ for $k\in\N$, or Cohen-Macaulay if and only if $S$ is so or $\sig=\emptyset$; it is irreducible or integral if and only if $S$ is so and $\sig\neq\emptyset$; it is artinian if and only if $S$ is artinian and $n=0$, or $S=\emptyset$, or $\sig=\emptyset$.

c) It holds $\dim(S)+n\leq\dim(\xs(S))\leq(n+1)\dim(S)+n$; if $S$ is locally noetherian then it holds $\dim(S)+n=\dim(\xs(S))$.

d) If $S$ is locally noetherian, then $\xs(S)$ is equidimensional if and only if $S$ is so or $\sig=\emptyset$; if $n\neq 0$ and $S$ is pointwise noetherian, then $\xs(S)$ is (universally) catenary if and only if $S$ is universally catenary or $\sig=\emptyset$.

e) There is a canonical isomorphism $\xs(\bullet)_{\red}\cong\xs(\bullet_{\red})$, and if $\sig\neq\emptyset$ then the maps $Z\mapsto\xs(Z)$ and $Z\mapsto\ts(S)(Z)$ induce mutually inverse bijections between the sets of irreducible or connected components of $S$ and $\xs(S)$.
\end{thm}

\begin{proof}
Immediately from the results in Section 2.
\end{proof}

The last properties treated here are Serre's properties $(R_k)$ and regularity. They turn out to behave in an essentially different way than the above properties: they depend on the fan $\sig$. This is known in the classical case, and this behaviour is also known in general for properness: By \cite[4.4]{dem} (where the hypothesis that $\sig$ is $N$-regular is not needed), $\ts(S)$ is proper if and only if $\sig$ is complete or $\sig=\emptyset$ or $S=\emptyset$. Recall that for $k\in\N$ a scheme $X$ is said to have property $(R_k)$ if it is locally noetherian and $\mathscr{O}_{X,x}$ is regular for every $x\in X$ with $\dim(\mathscr{O}_{X,x})\leq k$, and a morphism of schemes $X\rightarrow Y$ is said to have property $(R_k)$ if it is flat and $X_z\otimes_{\kappa(z)}K$ has $(R_k)$ for every $z\in Y$ and every finite extension $\kappa(z)\rightarrow K$. Furthermore, a scheme or morphism of schemes is called regular if it has $(R_k)$ for every $k\in\N$ (\cite[IV.5.8.1--2; IV.6.8.1]{ega}).

\begin{thm}\label{thm2}
The following statements are equivalent: (i) $\ts(S)$ is regular; (ii) $\ts(S)$ has $(R_n)$; (iii) $\sig$ is $N$-regular or $S=\emptyset$.
\end{thm}

\begin{proof}
If $\ts(S)$ has $(R_n)$ and $S\neq\emptyset$ then there exists a field $K$ such that $\xs(K)$ has $(R_n)$. As $\dim(\xs(K))=n$ (\ref{thm1}) this implies that $\xs(K)$ is regular, and then it is well-known that $\sig$ is $N$-regular (e.g.~\cite[3.1.19]{cls}). So, (ii) implies (iii). If $\sig$ is $N$-regular then it is well-known that $\xs(K)$ is regular for every field $K$ (e.g.~\cite[3.1.19]{cls}), and since $\ts(S)$ is flat (\ref{thm1}) it follows that $\ts(S)$ is regular. Hence, (iii) implies (i), and thus the claim is proven.
\end{proof}

\begin{cor}
a) Let $k\in\N$. If $S$ has $(R_k)$ and $\sig$ in $N$-regular, or $S=\emptyset$, or $\sig=\emptyset$, then $\xs(S)$ has $(R_k)$; if $k\geq n$ then the converse holds.

b) $\xs(S)$ is regular if and only if $S$ is regular and $\sig$ is $N$-regular, or $S=\emptyset$, or $\sig=\emptyset$.
\end{cor}

\begin{proof}
Immediately from \ref{thm1}, \ref{thm2} and \cite[IV.6.5.3]{ega}.
\end{proof}


\smallskip

\noindent\textbf{Acknowledgement.} I thank Markus Brodmann, Stefan Fumasoli, J\"urgen Hausen, Kazuma Shimomoto, and {\fontencoding{T5}\selectfont Ng\ocircumflex{} Vi\d\ecircumflex{}t Trung} for their various help, and the referee for his careful reading and his comments and suggestions.



\begin{thebibliography}{99}
\bibitem{arnold} J. T. Arnold, R. Gilmer, \emph{The dimension sequence of a commutative ring.} Amer. J. Math. 96 (1974), 385--408.
\bibitem{a} N. Bourbaki, \emph{\'El\'ements de math\'ematique. Alg\`ebre. Chapitres 1 \`a 3.} Hermann, Paris, 1970.
\bibitem{tg} N. Bourbaki, \emph{\'El\'ements de math\'ematique. Topologie g\'en\'erale. Chapitres 1 \`a 4.} Hermann, Paris, 1971.
\bibitem{ac} N. Bourbaki, \emph{\'El\'ements de math\'ematique. Alg\`ebre commutative. Chapitres 1 \`a 4.} Masson, Paris, 1985; \emph{Chapitres 5 \`a 7.} Hermann, Paris, 1975; \emph{Chapitres 8 et 9.} Masson, Paris, 1983.
\bibitem{bg} W. Bruns, J. Gubeladze, \emph{Polytopes, rings, and K-theory.} Springer Monogr. Math., Springer, Berlin, 2009.
\bibitem{chu} C. Chu, O. Lorscheid, R. Santhanam, \emph{Sheaves and K-theory for $\mathbbm{F}_1$-schemes.} Adv. Math. 229 (2012), 2239--2286.
\bibitem{cc1} A. Connes, C. Consani, \emph{Schemes over $\mathbbm{F}_1$ and zeta functions.} Compos. Math. 146 (2010), 1383--1415.
\bibitem{cc3} A. Connes, C. Consani, \emph{On the notion of geometry over $\mathbbm{F}_1$.} J. Algebraic Geom. 20 (2011), 525--557.
\bibitem{cc2} A. Connes, C. Consani, M. Marcolli, \emph{Fun with $\mathbbm{F}_1$.} J. Number Theory 129 (2009), 1532--1561.
\bibitem{chww} G. Corti\~nas, C. Haesemeyer, M. E. Walker, C. Weibel, \emph{Toric varieties, monoid schemes and cdh descent.} Preprint (2011). {\tt arXiv:1106.1389}.
\bibitem{cls} D. Cox, J. Little, H. Schenck, \emph{Toric varieties.} Grad. Stud. Math. 124, Amer. Math. Soc., Providence, 2011.
\bibitem{deitmar1} A. Deitmar, \emph{Schemes over $\mathbbm{F}_1$.} In: Number fields and function fields -- two parallel worlds (Eds. G. van der Geer, B. Moonen, R. Schoof), Progr. Math. 239, Birkh\"auser, Basel, 2005; 87--100.
\bibitem{deitmar2} A. Deitmar, \emph{Remarks on zeta functions and K-theory over $\mathbbm{F}_1$.} Proc. Japan Acad. Ser. A Math. Sci. 82 (2006), 141--146.
\bibitem{dem} M. Demazure, \emph{Sous-groupes alg\'ebriques de rang maximum du groupe de Cremona.} Ann. Sci. \'Ec. Norm. Sup\'er. (4) 3 (1970), 507--588.
\bibitem{ful} W. Fulton, \emph{Introduction to toric varieties.} Ann. of Math. Stud. 131, Princeton Univ. Press, Princeton, 1993.
\bibitem{gilmer} R. Gilmer, \emph{Commutative semigroup rings.} Chicago Lectures in Math., Univ. Chicago Press, Chicago, 1984.
\bibitem{ega} A. Grothendieck, J. A. Dieudonn\'e, \emph{\'El\'ements de g\'eom\'etrie alg\'ebrique. I: Le langage des sch\'emas (seconde \'edition).} Grundlehren Math. Wiss. 166, Springer, Berlin, 1971; \emph{II: \'Etude globale \'el\'ementaires de quelques classes de morphismes.} Publ. Math. Inst. Hautes \'Etudes Sci. 8 (1961), 5--222; \emph{IV: \'Etude locale des sch\'emas et des morphismes de sch\'emas II.} Publ. Math. Inst. Hautes \'Etudes Sci. 24 (1965), 5--231.
\bibitem{gubler} W. Gubler, \emph{A guide to tropicalizations.} Preprint (2011). {\tt arXiv:1108.6126}.
\bibitem{hochster} M. Hochster, \emph{Rings of invariants of tori, Cohen-Macaulay rings generated by monomials, and polytopes.} Ann. of Math. (2) 96 (1972), 318--337.
\bibitem{kkms} G. Kempf, F. F. Knudsen, D. Mumford, B. Saint-Donat, \emph{Toroidal embeddings I.} Lecture Notes in Math. 339, Springer, New York, 1973.
\bibitem{map} J. L\'opez-Pe\~na, O. Lorscheid, \emph{Mapping $\mathbbm{F}_1$-land: an overview of geometries over the field with one element.} In: Noncommutative geometry, arithmetic, and related topics (Eds. C. Consani, A. Connes), Johns Hopkins Univ. Press, Baltimore, 2012; 241--265.
\bibitem{mat} H. Matsumura, \emph{Commutative ring theory.} Translated from the Japanese. Cambridge Stud. Adv. Math. 8, Cambridge Univ. Press, Cambridge, 1986.
\bibitem{ohm-pen} J. Ohm, R. L. Pendleton. \emph{Rings with noetherian spectrum.} Duke Math. J. 35 (1968), 631--639.
\bibitem{ratliff} L. J. Ratliff Jr., \emph{On quasi-unmixed local domains, the altitude formula, and the chain condition for prime ideals (II).} Amer. J. Math. 92 (1970), 99--144.
\bibitem{diss} F. Rohrer, \emph{Toric schemes.} Dissertation, Universit\"at Z\"urich, 2010.
\bibitem{tv} B. To\"en, M. Vaqui\'e, \emph{Au-dessous de ${\rm Spec}\,\mathbbm{Z}$.} J. K-Theory 3 (2009), 437--500.
\bibitem{vezzani} A. Vezzani, \emph{Deitmar's versus To\"en-Vaqui\'e's schemes over $\mathbbm{F}_1$.} Math. Z. 271 (2012), 911--926.
\end{thebibliography}
\end{document}